 \documentclass[11pt,english]{article}

\usepackage[T1]{fontenc}
\usepackage[latin1]{inputenc}
\usepackage{babel}

\usepackage{setspace}
\usepackage{indentfirst}

\makeatletter

\usepackage{verbatim}

\makeatother

\usepackage{enumerate}
\usepackage{amsmath}
\usepackage{amssymb}
\usepackage{stmaryrd}
\usepackage{wasysym}
\usepackage{pifont} 
\usepackage{graphicx}
\usepackage[all]{xy}

\usepackage{amsthm}
\theoremstyle{plain}
\newtheorem{theo}{Theorem}[section]
\newtheorem{lem}[theo]{Lemma}

\newtheorem{convention}[theo]{Convention}

\newtheorem{prop}[theo]{Proposition}
\newtheorem{coro}[theo]{Corollary}
\newtheorem{defn}[theo]{Definition}
\theoremstyle{definition}
\newtheorem{example}[theo]{Example}
\theoremstyle{remark}
\newtheorem{rmk}[theo]{Remark}

\renewcommand{\epsilon}{\varepsilon}

\def\dar[#1]{\ar@<2pt>[#1]\ar@<-2pt>[#1]}

\newcommand{\toto}{\rightrightarrows}

\newcommand{\be }{\begin{eqnarray*}}
\newcommand{\ee }{\end{eqnarray*}}

\title{Equivariant resolutions of singularities as differential substacks.}

\author{Camille Laurent-Gengoux  \\
 \\
D\'epartement de math\'ematiques \\
Universit\'e
de Poitiers \\
86962 Futuroscope-Chasseneuil, France \\
\texttt{laurent@math.univ-poitiers.fr} }

\begin{document}
\maketitle

\begin{abstract}
We show that there is an one-to-one correspondence between resolutions (equivariant w.r.t. a Lie groupoid action) of a singular subset 
of a manifold, and substacks (of a certain type) of the differential
stack associated to the Lie groupoid in question. In particular, we
show how to build a resolution out of Lie subgroupoids (of a certain type). 
\end{abstract}

\tableofcontents

\section{Introduction}
 
Differential stacks are, from the very beginning, a way to get rid of
singularities (of a quotient space, of a foliation and so on).
This article intends to convince the reader that there is also something in common between differential stacks
and  resolutions of singularities where "resolutions" has to be taken in
the sense it has in Hironaka's Big Theorem.

We should however confess that we are not able to say anything interesting in the very
general case of arbitrary resolutions of singularities, but we claim 
to have a non-trivial classification result for the case of resolutions equipped with some additional
symmetry  (like a group action, a Poisson or a symplectic structure.)
Our precise claim is that equivariant resolutions of singularities included in the unit manifold
of a Lie groupoid $\Gamma $  are classified by some classes of substacks of the differential stacks associated to $\Gamma $.
(Exactly as differential stacks are Lie groupoids modulo Morita equivalence, differential substacks are Lie subgroupoids modulo Morita
equivalence). Working out this correspondence more accurately, we give a dictionary
between the properties of the equivariant resolutions and the properties of its corresponding
substacks. More precisely,  we give a dictionary between the properties of the 
equivariant resolutions and the corresponding subgroupoids, whose quotient modulo Morita equivalence form the substack in question.

The first issue that one has to face is that resolutions of singularities and differential stacks do not belong to
the same branch of mathematics: resolutions of singularities form a chapter of algebraic geometry,
while  differential stacks are objects within  differential geometry (real or complex).
Indeed, we shall suggest a reasonable notion of ``resolution of singularities'' in the setting of
differential geometry, which contains the algebraic ones (over ${\mathbb R} $ or ${\mathbb C} $) as a
particular case. The correspondence that we are then going to
establish is between some well-chosen substacks of differential stacks 
and equivariant resolutions of singularities, as now defined in differential geometry. 

One may argue that it would be more interesting to remain inside algebraic geometry, and to deal with algebraic stacks, or rather, with algebraic groupoids. The author totally agrees with this objection, and would like simply the reader to allow him to postpone this study to a future work.

The author also totally agrees with the fact the present article lacks of convincing examples. But they exist, of course. As shown in \cite{desing},  symplectic resolutions, as defined by Beauville \cite{Beau}, as-well as Poisson resolutions, as defined by Baohua Fu \cite{FuSum,Fu}, form  classes of examples of equivariant resolutions,(equivariant w.r.t. a symplectic Lie groupoid), which are constructed out of Lie subgroupoids. Indeed, these symplectic resolutions form the initial motivation: the present work should be considered both as a preliminary to \cite{desing} (where these examples are studied), but also as an answer to some question raised in \cite{desing}. In particular, it gives a way to determine whether or not the resolutions associated to two different Lie subgroupoids are isomorphic as resolutions.  


The paper is organized as follows. In section \ref{sec:R1}, the object that we are going to desingularize,
namely closure of $\Gamma $-invariant submanifolds, is introduced, together with the corresponding notion of resolution of singularities,
namely equivariant resolutions. More precisely, these objects are introduced in \ref{sec:R1_D1}, and 
their good  behavior under Morita equivalence is studied in \ref{Morita-resolution}.
Substacks and their various subclasses are introduced in section \ref{Substacks}.
The correspondences between equivariant resolutions and some classes of substacks is then detailed in
section \ref{the_correspondence}, a section entirely devoted to the proof our main result, namely Theorem \ref{fund:theo}.

We will always make use of the following convention about Lie groupoids:

\begin{convention}
The notation most often used to denote a Lie groupoid is $ \Gamma \toto M$, a convenient way to provide the reader 
the names of both sets of arrows and objects, and which underlines on the role of source and target maps,
represented by the two parallel arrows. 
The shorthand $\Gamma$ may be used instead of $ \Gamma \toto M$.
Of course, the structural maps $s,t,\epsilon, \mu,{\rm inv} $ (i.e. source, target, unit, product, inverse) are not explicitly referred to in that notation, but this ambiguity shall never be an issue, since, anyway, we never have to consider two different groupoid structures on the same pair of sets $(\Gamma,M)$, which allows to use the notations $s,t,\epsilon, \mu,{\rm inv} $ to denote the structural maps of {\em all} Lie groupoids that we shall meet in the sequel. For a given groupoid $\Gamma \toto M $, and for $I,J \subset M$, we introduce $\Gamma_I:= s^{-1}(I), \Gamma^J:=t^{-1}(J)  $, and $\Gamma_I^J:= \Gamma_I \cap \Gamma^J $. When $I=\{x\}$ or $ J=\{y\}$ reduce to a point, the shorthands $\Gamma_x , \Gamma^y, \Gamma_x^y$ will be used instead of $\Gamma_{\{x\}} , \Gamma^{\{y\}}, \Gamma_{\{x\}}^{\{y\}}$. 

Also, $M$ will be most of the time considered as a submanifold of $\Gamma$ (in particular, no notational distinction
between $m \in M$  and $\epsilon(m) \in \Gamma$ will be made).
\end{convention}

\section{Equivariant resolution of the closure of $\Gamma $-invariant submanifolds.}
\label{sec:R1}

\subsection{Definition of an equivariant resolution in differential geometry}
 \label{sec:R1_D1}

In the context of algebraic geometry, a {\em resolution} of an affine or projective variety $W$ is pair $(Z,\phi)$ where $Z$ is a smooth (= without singularities) variety, and $\phi: Z \to W $  a regular map from $Z$ to $W$, whose restriction to $\phi^{-1}(W_{reg }) $ is an isomorphism onto $W_{reg}$ (here, $W_{reg}$ stands for the regular part). We shall say that this resolution is {\em surjective} if $\phi (Z) =W$, and {\em proper} if $\phi$ is a proper map. Notice that a proper resolution is always surjective. Notice also that surjectivity 
or properness is often taken as part of the definition  in the literature.

In the context of differential geometry, we suggested in \cite{desing}
to mimic the previous requirements as follows: what plays the role of
$W$ is the closure $\bar{\mathcal S} $ of an embedded submanifold
${\mathcal S} $ in $W$, the role of the regular part $W_{reg} $ being
then played by ${\mathcal S} $ itself. We shall try to desingularize
these objects. By the resolution of the later, in view of the
analogous algebraic case, we suggested in \cite{desing} the following:

\begin{defn} \label{def:desing}
Let $\bar{{\mathcal S}}$ be  the closure of an embedded submanifold ${\mathcal S} $ of a 
complex/real manifold $ M$.  A {\em resolution of $\bar{{\mathcal S}} $} is a pair $(Z,\phi)$ where $Z$ is a
complex/real manifold and $\phi:Z \to M$ is a holomorphic/smooth map such that
\begin{enumerate} 
\item $\phi^{-1}({\mathcal S}) $ is dense in $Z$,
\item the restricted map $\phi :\phi^{-1}({\mathcal S}) \to {\mathcal S}$ is an biholomorphism/diffeomorphism.
\end{enumerate}
We say that this resolution is {\em surjective} if $\phi(Z) = {\bar{\mathcal S}}$ and {\em proper} 
if the map $\phi $ is proper.
\end{defn}
%

Note that, the submanifold ${\mathcal S}$ being assumed to be embedded, ${\mathcal S}$
is an open subset of $\bar{{\mathcal S}} $. We immediately connect this new notion with the traditional one.

\begin{prop}
A resolution (resp. surjective resolution / proper resolution), in the sense of  algebraic geometry over the field
${\mathbb C} $ or ${\mathbb R} $, of an irreducible affine
variety $W \subset {\mathbb C}^N$ or $ W \subset 
{\mathbb R}^N $, is a resolution (resp. surjective resolution / proper resolution), in the sense of
Definition \ref{def:desing} taken in the holomorphic or smooth context, of  the closure of $W_{reg} $.
\end{prop}
\begin{proof}
To start with, notice that $W_{reg} $ is a (smooth / holomorphic) submanifold of $ {\mathbb R}^N $ or $ {\mathbb C}^N $, with closure $\overline{W_{reg}} =W $ (the closure being taken w.r.t. the usual topology).
Let $(Z,\phi)$ be a resolution in the sense of algebraic
geometry. First, $\phi^{-1}(W_{reg}) $ is a non-empty Zariski open
subset, and is therefore dense for the usual topology  in $Z$. Second,
the restricted map $\phi :\phi^{-1}({\mathcal S}) \to {\mathcal S}$ is
biregular, it is therefore also a biholomorphism or the smooth map, depending on the base field. 
Hence it forms a resolution in the sense of Definition~\ref{def:desing}. 

Surjectivity of the resolution has exactly the same meaning in algebraic, holomorphic or smooth context.
Moreover, if $(Z,\phi) $ is proper is the sense of algebraic geometry,
then it is also proper with respect to the usual topology, see \cite{Sha} section I.5.2.
\end{proof}

A {\em morphism} from a resolution $(Z_1,\phi_1) $ to $(Z_2,\phi_2) $ is a map $\Phi: Z_1 \to Z_2 $ such
that $\phi_2 \circ \Phi = \phi_1 $. The restriction of $\Phi $ to $\phi_1^{-1}({\mathcal S}) $ coincides with
$ \phi_2^{-1} \circ \phi_1 $. By density of $\phi_1^{-1}({\mathcal S}) $ in $Z_1 $, if two resolutions are
isomorphic, then there is  {\em one and exactly one} isomorphism between them. This justifies the following convention.

\begin{convention}\label{con:isom}
From now, we shall identify two isomorphic resolutions.
\end{convention}

Singular spaces that we shall be interested in are $\Gamma $-stable
submanifolds of the base manifold $M$ of a Lie groupoid $\Gamma \toto
M $. More precisely, a subset ${\mathcal S} \subset M$ is said to be {\em
  $\Gamma$-stable} if and only if 
for all $\gamma \in \Gamma$
$$ t(\gamma) \in {\mathcal S} \hspace{0.5cm} \Leftrightarrow \hspace{0.5cm} s(\gamma) \in {\mathcal S} .$$
Equivalently, a subset ${\mathcal S} \subset M$ is $\Gamma$-stable if and only if it is a disjoint union of Lie groupoid orbits.
We start with a Lemma.

\begin{lem}\label{lem:leaf}
Let $\Gamma \toto M$ be a Lie groupoid. For all embedded $\Gamma $-stable submanifold ${\mathcal S} $ of
$M$, its closure $\bar{{\mathcal S}}$ is $\Gamma$-stable.
\end{lem}
\begin{proof}
For every $m \in \bar{{\mathcal S}} $, $\gamma \in \Gamma$ with $t(\gamma)=m$, there exists
a neighborhood $U$ of $ m$ and a local section $\sigma: U \to \Gamma $
 through $\gamma \in \Gamma$ of the target map $t: \Gamma \to M$.
Let $(u_n)_{n \in {\mathbb N}}$ be a sequence of points in ${\mathcal S} $ converging to $m $, then 
the sequence $v_n : n \mapsto  s \circ \sigma (u_n)$ is a sequence converging to $s (\gamma) $.
Since ${\mathcal S} $ is $\Gamma$-stable, the sequence $(v_n)_{n \in {\mathbb N}}$ takes values in ${\mathcal S}$,
and $s(\gamma)$ belongs to $ \bar{{\mathcal S}} $.
\end{proof}

Recall that a left (resp. right) $\Gamma $-module is a pair $(Z,\phi) $, with $Z$ a
manifold, and $\phi: Z \to M $ a map, endowed with a left (resp. right) action  of 
$ \Gamma \toto M$, i.e. a map from $Z \times_{\phi,M,s}\Gamma \to Z  $ (resp. $\Gamma \times_{t,M,\phi} Z \to Z $) satisfying some natural axioms, see e.g. \cite{McK}. 

\begin{example}\label{ex:trivialGammaModules}
A trivial but however important example of $\Gamma
$-module are the pairs $({\mathcal S}, {\mathfrak i}_{\mathcal S})$, with ${\mathcal S}
\subset M$ a $\Gamma$-stable submanifold and ${\mathfrak i}_{\mathcal S} : {\mathcal S} \hookrightarrow M$ the inclusion
map.  \end{example}

Notice that both resolutions of  $ \overline{\mathcal S} $ and $\Gamma$-modules
consist in pairs $(Z,\phi) $, with $Z $ a manifold, and $\phi: Z \to M$
a map. This leads to the following natural definition.

\begin{defn} \label{def:desingequiv}
Let $\Gamma \toto M $ be a Lie groupoid, and ${\mathcal S} $ a $\Gamma $-stable
embedded submanifold of $ M$. A {\em $\Gamma$-resolution of
  $\overline{\mathcal S}$} or   is a pair
$(Z,\phi) $, with $Z $ a manifold, and $\phi: Z \to M$
a map, which 
\begin{enumerate}
\item admits a structure of right $\Gamma $-module, and
\item is a resolution of $\overline{\mathcal S}$.
\end{enumerate}
We shall speak of an {\em equivariant resolution} instead of $\Gamma$-resolution when we do not to emphasize
on the name of the Lie groupoid.
\end{defn}
\begin{example}
Let $G$ be a Lie group acting on a manifold $M$.  Recall from \cite{McK} that $\Gamma = G \times M \toto M $ admits a natural   Lie groupoid
called the {\em transformation groupoid}. With respect to this Lie groupoid, 
$\Gamma$-stable submanifolds are submanifold stable under the action of 
$G$, and $\Gamma $-resolutions of $ \overline{\mathcal S}$ are resolutions which are equivariant
w.r.t.  the action of~$G$.
\end{example}

\begin{rmk}
By definition of a resolution of $\overline{\mathcal S} $, with ${\mathcal S}$
a $\Gamma$-stable submanifold, $\phi: \phi^{-1}({\mathcal S}) \to {\mathcal S} $ is a bijective map,
so that  the restriction of the $\Gamma $-action to $\phi^{-1}({\mathcal S})$
needs to be given for all $z \in \phi^{-1}({\mathcal S}) $ and
$\gamma \in \Gamma_{\phi(z)}$ by $ z \cdot \gamma = \phi^{-1} (s(\gamma)) .$
Since $\phi^{-1}({\mathcal S}) $ is a dense subset of $Z$, 
there is at most {\rm one} structure of $\Gamma$-module on a given resolution of $\overline{\mathcal S} $. 
Indeed, one could define $\Gamma$-resolutions of  $\overline{\mathcal S}$
as being those for which the natural action on $\Gamma$ on $\phi^{-1}({\mathcal S}) \subset Z$ 
extends to an action on $Z$ (in a smooth or holomorphic way). 
\end{rmk}
\subsection{Morita equivalence and equivariant-resolutions}
\label{Morita-resolution}

\vspace{0.5cm}
{\bf a) Morita equivalence, differential stacks.}
We briefly recall the definition of Morita equivalence of Lie groupoids, as constructed with the help of bimodules. 

\begin{defn}\label{def:Morita}
Let  $\Gamma \toto M $ and $\Gamma' \toto M' $ be two Lie groupoids. A $\Gamma-\Gamma'$-{\em bimodule} is a manifold $X $ endowed with two surjective submersions $ p: X \to M$  and $p': X \to M' $, so that
\begin{enumerate}
\item  $(Z,p) $ endows a structure of  left $\Gamma $-module; 
\item $(Z,p') $ endows a structure of right $\Gamma' $-module;
\item the right (resp. left) actions preserves all the fibers of $ p$ (resp. $p' $); 
\item The right and left actions commute, i.e. for all $\gamma \in \Gamma, z \in Z, \gamma' \in \Gamma' $
with $ t(\gamma) = p (z)  $ and $  p' (z) = s(\gamma') $:
 $$  (\gamma \cdot z) \cdot \gamma' = \gamma \cdot (z \cdot \gamma').$$
\end{enumerate}
A {\em Morita equivalence} is a $\Gamma $-$\Gamma'$-bimodule such that the right and left actions are both proper, free and transitive on the fibers of $p'$ and $p$ respectively.
\end{defn}
\begin{convention}
It shall be convenient to denote a Morita equivalence simply by a ${\mathcal X}$, although this notation does not make explicit the many structures it is equipped with. More precisely, from now, we shall denote by a curvy letter
 ${\mathcal X}$ (or ${\mathcal X}'$ / or ${\mathcal Y}$...), a Morita equivalence given by a set $X $ (or $X'$ / or $Y$...). In all the cases, $p$ and $p'$ shall stand for the two maps from $X$ (or $X' $ / or $Y $...) to the unit manifolds of the two Lie subgroupoids. Also, in all the cases, both left and right actions shall be simply denoted by a dot "$\cdot$". 
\end{convention}

Recall from \cite{LTX} that Morita equivalences can be composed, and that this composition is associative (up to isomorphism).
We briefly recall the construction. Let ${\mathcal X} $ and ${\mathcal X}' $ be Morita equivalences between Lie groupoids
$\Gamma_1 \toto M_1 $, $\Gamma_2 \toto M_2$ and  $\Gamma_3 \toto M_3$. Then  
the following data define a Morita equivalence ${\mathcal X}'' $ between $\Gamma_1 \toto M_1$ and  $\Gamma_3 \toto M_3$.
  \begin{enumerate}
  \item the manifold $X'' = \frac{X \times_{p',M_2,p} X'}{ \Gamma_2 } $,
  where  $\Gamma_2 \toto M_2 $ acts on $ X \times_{p',M_2,p} X' $ by  $\gamma_2 \cdot (x,x')=(x \cdot (\gamma_2)^{-1}, \gamma' \cdot x) $ for all $\gamma_2 \in \Gamma_2$, $x \in X $, $x' \in X'$ with $p'(x)=p(x') = t(\gamma') $.
   \item the maps  $p([(x,x')])=p (x)$ and $p'([(x,x')])=p' (x')$, where $[(x,x')]$ stands for the class of $(x,x')
   \in  X \times_{p_2',M_2,p_2} X'$   modulo the action of $\Gamma_2$,
  \item the right and left actions given by $\gamma_1 \cdot [(x,x')] \cdot \gamma_3 = [(\gamma_1 \cdot x,x' \cdot \gamma_3)]
   $, for all $\gamma_1 \in \Gamma_1 , \gamma_3 \in \Gamma_3$ with $ t(\gamma_1) = p(x), s(\gamma_3) = p'(x') $. 
  \end{enumerate}

\begin{convention}
From now, we shall identify isomorphic Morita equivalences, so that the composition of those becomes associative.
\end{convention}

A {\em differential stack}  is an equivalence class of Lie groupoids
modulo Morita equivalence. Given a Lie groupoid $\Gamma$, $[\Gamma]$
stands for the differential stacks to which it belongs.

\begin{rmk} One may argue that
Lie groupoids do not form a set, so that the terminology "equivalence relation" should be banned, 
and that the language of category would be more accurate here.
However, it would lead to unnecessary sophisticated complications, that we prefer to avoid, but
we shall in the sequel restate some of our results using categorical language as a remark.
\end{rmk}

 By a {\em representative} of the differential stack $[\Gamma]$, we mean  a pair formed by a Lie groupoid $\Gamma'$ together with a Morita equivalence ${\mathcal X} $ between $\Gamma$ and $\Gamma' $.

\begin{rmk} We warn the reader of a possible confusion: a representative of $[\Gamma]$
is not simply a Lie groupoid which happens to be Morita equivalent to
$\Gamma$, but a Lie groupoid  Morita equivalent to
$\Gamma$ {\em together with} a Morita equivalence relating it to $\Gamma$.
\end{rmk}

In the rest of this section, we review or study why Morita equivalence induces one-to-one correspondence of modules,
stable subsets, closure of stable submanifolds, equivariant resolutions.

\vspace{0.5cm}

{\bf b) Morita equivalence induces one-to-one correspondence of right-modules.}
It is well-known that a Morita equivalence ${\mathcal X}$ between
Lie groupoids $\Gamma $ and $\Gamma' $ induces a one-to-one
correspondence $\underline{\mathcal X} $ between right $\Gamma
$-modules and right $\Gamma' $-modules. This fact  is stated in the
present form in \cite{LTX} and is more of less implicit in \cite{Xu}, but we prefer to briefly recall the construction
that we shall use several times.
 For $(Z,\phi) $ a right $\Gamma $-module, $
\underline{\mathcal X} \big( (Z,\phi) \big) $ is the right $\Gamma' $-module $(Z',\phi') $ defined by
\begin{enumerate}
\item the manifold $ \frac{Z \times_{\phi, M, p} X }{ (z,x)\sim (z
  \gamma^{-1}, \gamma x ) }$ (which is a manifold, because, first, $Z
  \times_{\phi, M, p} X $ is itself a manifold since $p$ is a
  submersion, and, second, because the left action of $\Gamma $ on $X $ being a free
  and proper action,  so is its action on $Z \times_{\phi, M, p} X $, so that the quotient is a manifold), 
\item the map $[z,x] \mapsto p' (x) $, where $[z,x] $ is the class of $(z,x) \in Z \times_{\phi, M, p} X $
modulo the action of $\Gamma \toto M $, 
\item the right action defined by $  [z,x] \cdot \gamma' = [z, x \cdot \gamma'] $, for all $\gamma' \in s^{-1} (p'(x)) $,
and all $(z,x) \in Z \times_{\phi, M, p} X $.
\end{enumerate}

\begin{rmk}\label{rmk:functM}
To any morphism $\phi: Z_1 \to Z_2$ of right $\Gamma $-module is associated a morphism of the corresponding
right $\Gamma' $-module $\underline{\mathcal X}(\phi): \underline{\mathcal X}(Z_1)\to \underline{\mathcal X}(Z_2) $.
Said differently, the previous construction is functorial.
In particular, it would be more rigorous to use the language of categories and to claim that a Morita equivalence induces an equivalence of categories between  the categories of right $\Gamma $-modules
and right $\Gamma' $-modules.
\end{rmk}

\vspace{0.5cm}
{\bf c) Morita equivalence induces one-to-one correspondence of stable subsets.}
The Morita equivalence ${\mathcal X}$ also induces a
one-to-one correspondence between $\Gamma $-stable subsets of $M $ and
$\Gamma' $-stable subsets of $ M'$ by 
\begin{equation}\label{eq:subsets}S \mapsto p' \big( p^{-1}(S)
\big) . \end{equation}
The previous correspondence maps in particular orbits to orbits, and $\Gamma
$-stable submanifolds to $\Gamma'$-stable submanifolds. But $\Gamma $-stable submanifolds are also right $\Gamma$-modules
(as seen in example \ref{ex:trivialGammaModules}), and should be mapped by $ \underline{\mathcal X} $ to a right
$\Gamma'$-module. In both both constructions agree, in the sense that 
$$\underline{\mathcal X} \big( ({\mathcal S},{\mathfrak i}_{\mathcal S}) \big) = (
 {\mathcal S}', {\mathfrak i}_{{\mathcal S}'})  ,$$
where ${\mathcal S}' = p' \big( p^{-1}({\mathcal S}) \big) $.
In words: ``The restriction of $\underline{\mathcal X}$ to $\Gamma $-stable submanifolds is the correspondence  given by Equation (\ref{eq:subsets})''. This allows one to state the following convention:
  
\begin{convention}  From now, given a Morita equivalence $ {\mathcal X}$ between $\Gamma$
and $\Gamma'$, we denote by the same symbol $\underline{\mathcal X} $  the
one-to-one correspondences between $\Gamma$-stable and $\Gamma'$-stable subsets given by Equation (\ref{eq:subsets}),
and the one-to-one correspondences between right $\Gamma$-modules and right $\Gamma'$-module described in {\rm {\bf b)}} above. 
\end{convention}

\vspace{0.5cm}
{\bf The previous correspondence is compatible with closure.}
We would like the reader to understand the next lemma as follows: ``Morita equivalence
 preserves the shape of the closure of $\Gamma$-stable submanifolds''.

\begin{lem}\label{lem:comm_clos}
Let ${\mathcal X}$ be a Morita equivalence between $\Gamma \toto M$ and $\Gamma' \toto M'$, and
let ${\mathcal S} $ be a $\Gamma$-stable submanifold of $M$. Then
$$ \underline{\mathcal X} \big( \overline{\mathcal S} \big) =   \overline{   \underline{\mathcal X} ({\mathcal S} ) }.$$
\end{lem}
\begin{proof}
We prefer to give the proof in great detail, although it is just an undergraduate exercise in topology.
For every $m' \in  \underline{\mathcal X} \big( \overline{\mathcal S} \big)$, there exists 
an element $ x \in X$ such that $p(x) \in  \overline{\mathcal S} $ and $p'(x)=m'$.
Moreover, there exists a sequence $ (u_n)_{n \in {\mathbb N}}$ of elements in $ {\mathcal S}$ converging to $m \in M  $. Since $p: X \to M $ is a submersion, there exist a neighborhood $U$ of $m$ and 
a local section $\sigma: U \to X $ of $p $ through $ x \in X $.
The sequence $n \mapsto \sigma(u_n)$ takes its values in $p^{-1}({\mathcal S}) $
and converges to $x$. Hence the sequence $ n \mapsto p'(\sigma(u_n))$ 
takes its values in $p'(p^{-1}({\mathcal S}))= \underline{\mathcal X} ({\mathcal S} ) $
and converges to $m'$, so that we obtain the inclusion
 \begin{equation}\label{eq:incl1} \underline{\mathcal X} \big( \overline{\mathcal S} \big) \subset   \overline{   \underline{\mathcal X} ({\mathcal S} ) } . \end{equation}
 Conversely, for every $m' \in \overline{   \underline{\mathcal X} ({\mathcal S} ) }$,
and every $x \in X$ with $p'(x) = m' $, there exists a neighborhood $U'$ of $m'$ in $M$ 
and local section $\sigma': U' \to X$ of $p'$ through $x $. Let $(u_n')_{n \in {\mathbb N}}$  
in ${\mathcal X}({\mathcal S})$ converging to $m' $. The sequence $n \mapsto p ( \sigma' (u_n')) $ belongs to 
${\mathcal S} $, and converges to $p(x) $. Hence, $p(x) \in  \overline{\mathcal S}$, and
$m'= p'(x) $ in an element in $\underline{\mathcal X} \big( \overline{\mathcal S} \big)=p'\big(p^{-1}(\overline{\mathcal S})\big)$. 
In conclusion, we have the inclusion
\begin{equation}\label{eq:incl2}   \overline{\underline{\mathcal X} ({\mathcal S}) }  \subset   \underline{\mathcal X} \big( \overline{\mathcal S} \big) .\end{equation}
The lemma follows from  (\ref{eq:incl1})-(\ref{eq:incl2}).
\end{proof}

\begin{rmk}\label{rmk:incl}
In fact, we have proved that a $ \Gamma$-orbit ${\mathcal T}$ belongs to the closure of a 
$\Gamma $-orbit ${\mathcal S}$ if and only if $\underline{\mathcal X}
({\mathcal T}) $ belongs to the closure of $ \underline{\mathcal X} ({\mathcal S} ) $.
\end{rmk}

\vspace{0.5cm}
{\bf d) Morita equivalence induces one-to-one correspondence of equivariant resolutions.}
We would like the reader to understand the next proposition as follows: "The
notion of equivariant resolution goes down to the level of differential stacks".
\begin{prop}
Let ${\mathcal X} $ be a Morita equivalence between Lie groupoids $\Gamma \toto M $ and $\Gamma' \toto M' $.
Let  ${\mathcal S} $ be a $\Gamma$-stable submanifold of $M$, and
 ${\mathcal S}'= {\underline{\mathcal X}}({\mathcal S})  $ be the corresponding $\Gamma'$-stable submanifold of $M'$.

Then $ {\underline{\mathcal X}}$ restricts to a one-to-one correspondence between $\Gamma$-resolutions 
(resp. surjective $\Gamma $-resolutions / proper $\Gamma $-resolutions) of $\overline{\mathcal S}$ and
$\Gamma' $-resolutions (resp. surjective $\Gamma'$-resolutions /
proper $\Gamma'$-resolutions) of $\overline{{\mathcal S}'} $. 
\end{prop}
\begin{proof}
Let $(Z,\phi) $ be a resolution of $\overline{\mathcal S} $ and define ${\underline{\mathcal X}}(Z)=Z',{\underline{\mathcal X}}(\phi)=\phi'$.
Applying the functor $ {\underline{\mathcal X}}$ (see remark \ref{rmk:functM})  to the commutative diagram:
 $$  \xymatrix{ \phi^{-1}({\mathcal S})      \ar@{^{(}->}[r]^{i} \ar[d]^{\simeq }   & Z \ar[dl]^{\phi}   \\ {\mathcal S}& }  $$
(where the vertical arrow is an isomorphism) and using the functorial properties  of $\underline{\mathcal X} $, which maps ${\mathcal S} $ to $ {\mathcal S}' $, inclusion of modules to inclusion of modules, and isomorphisms of modules to isomorphisms of modules, one obtains the commutative diagram:
 $$  \xymatrix{ (\phi')^{-1}({\mathcal S}')  \ar@{^{(}->}[r]^{i} \ar[d]^{\simeq }   & Z' \ar[dl]^{\phi'}   \\ {\mathcal S}'& }  $$
 (where the vertical arrow is an isomorphism).
 In words, the restriction of $\phi' $ to $(\phi')^{-1}(\overline{{\mathcal S}'}) $ is an invertible map,
 hence ${\underline{\mathcal X}} $ maps resolutions of $\overline{{\mathcal S}} $ to resolutions of $\overline{{\mathcal S}'} $. 
 
 Using the functorial properties of $ {\underline{\mathcal X}}$ (see remark \ref{rmk:functM}) once more, we obtain:
   $$ \phi'(Z')=  {\underline{\mathcal X}} (\phi') \big(  {\underline{\mathcal X}}(Z) \big)= {\underline{\mathcal X}}(\phi(Z))  .$$
 Hence, if $(Z,\phi) $ is surjective, then:
  $$ \phi'(Z') =   {\underline{\mathcal X}}(\overline{\mathcal S}) = \overline{{\underline{\mathcal X}}({\mathcal S})} = 
   \overline{{\mathcal S}'} ,$$
 where Lemma \ref{lem:comm_clos} has been used to go from the second to the third equality.  
 In conclusion, ${\underline{\mathcal X}} $ maps surjective resolutions of $\overline{{\mathcal S}} $ to surjective resolutions 
 of~$\overline{{\mathcal S}'} $.  
 
Assume now that the resolution $(X,\phi)$ is proper, and let $K
\subset \overline{{\mathcal S}'} $ be a compact subset. There exists a
compact subset $\hat{K}$ in $X$ with $p'(\hat{K})=K $ (this is due to
the fact that $p'$ is a surjective submersion and therefore admits
local sections). By construction,b $(\phi^{'})^{-1} (K)$ is the image of $Z \times_{\phi, M, p} \hat{K} $ through the natural projection 
 $$   Z \times_{\phi, M, p} X  \mapsto \frac{Z \times_{\phi, M, p} X }{ (z,x)\sim (z \gamma^{-1}, \gamma x ) } = Z'.$$
Since $ \hat{K} $  is compact, so is $p( \hat{K})$, hence so is
$\phi^{-1}( p( \hat{K}) ) $ by properness of $\phi$. The compactness of $Z \times_{\phi, M, p} \hat{K} $ follows,
and implies in turn the compactness of its image $(\phi^{'})^{-1} (K)$.  
\end{proof}

\section{Substacks of differential stacks}
\label{Substacks}

A {\em Lie subgroupoid} is a pair $(\Gamma \toto M, R \toto L)$, with $\Gamma \toto M $
a Lie groupoid, $R $ a submanifold of $\Gamma $ and $L  $ a
submanifold of $M$ stable under the structural maps  (unit, source, target, multiplication
and inverse) of $\Gamma \toto M $.

\begin{defn}
A Lie subgroupoid $R \toto L $ of a Lie groupoid $\Gamma \toto M $ is said to be \emph{closed}
if $R $ is a closed subset in $\Gamma_L^L $. 
\end{defn}

\begin{rmk}
When $L$ is itself a closed submanifold of $M$, this condition simply amounts to request that $R$ is a closed subset of $\Gamma $. 
\end{rmk}


Let  ${\mathcal S} \subset M $ be a $\Gamma$-stable submanifold of $M$. 
A submanifold $ L $ of $ M $ is said to  {\em intersect transversally the $\Gamma $-orbits contained in ${\mathcal S}$} if for all $m \in L \cap {\mathcal S} $
 $$ T_m {\mathcal S} = T_m F_m + T_m L $$
where $F_m $ is the $\Gamma$-orbit through $ m \in M$.
 The sum is {\em not} assumed to be a direct sum in general. 

\begin{rmk}
The transversality condition, in terms of the Lie algebroid $A \to M $,
with anchor $\rho $, of the Lie groupoid $\Gamma  \toto M$, means that, for all $m \in L \cap {\mathcal S} $, we  have:
 $$ T_m{\mathcal S} = \rho_m (A_m) + T_mL.$$ 
 \end{rmk}

\begin{rmk}\label{rmk:trans}
More important is the following remark.
The transversality condition means that, for all $m \in L \cap
{\mathcal S} $, we  have that
 $ T_m{\mathcal S} $ is equal to the image $T_mt (T_m\Gamma_L)$ of $T_m\Gamma_L \subset T_m\Gamma$ 
 through the differential $T_mt$ of the target map $t: \Gamma \to M$ at $m\in M$.
 Said differently, for every $u \in T_m{\mathcal S}$, there exist a path $\epsilon \mapsto \gamma(\epsilon)$
 in $\Gamma_L$, starting at $m$, s.t.:
  $$ u= \left. \frac{{\rm d}}{\rm d\epsilon} \right|_{\epsilon=0} t(\gamma(\epsilon))   $$
 The same could be said of the source map, upon replacing $\Gamma_L$ by $\Gamma^L $.
 \end{rmk}

\begin{defn}\label{def:forsubgroup}
Let $\Gamma \toto M $ be a Lie groupoid and ${\mathcal S} \subset M $ be a $\Gamma$-stable submanifold. A Lie subgroupoid $R \toto L $ of $\Gamma \toto M $ is said to be {\em a Lie subgroupoid in ${\mathcal S} $} if 
\begin{enumerate}
\item[(a)] $L \cap {\mathcal S} $ is a dense subset of $L$,
\item[(b)] $L$ intersects transversally the $\Gamma $-orbits contained in ${\mathcal S}$,
\item[(c)] and $L$ has a non-empty intersection with all the $\Gamma $-orbits contained in~${\mathcal S}$. 
\end{enumerate}
 A Lie subgroupoid $R \toto L $ in ${\mathcal S}$ is said to be
\begin{enumerate}
\item {\em surjective in $\overline{\mathcal S}$} when
$L$ has a non-empty intersection with all the $\Gamma $-orbits contained in $\overline{\mathcal S}$,
\item {\em full in ${\mathcal S} $} when
 $$   R_{L \cap {\mathcal S}}^{L \cap {\mathcal S}} \, = 
 \, \Gamma_{L \cap {\mathcal S}}^{L \cap {\mathcal S}}  , $$
 (In other words: "an arrow in $ \Gamma$ connecting two points in $L \cap {\mathcal S}$ belongs to $R$")
\item {\em a proper subgroupoid} when for all compact subset $K
  \subset \overline{\mathcal S}$, the quotient topological space $R \backslash
  \Gamma^K_L $ is compact.
\end{enumerate}
\end{defn}

\begin{rmk}
If ${\mathcal S} $ is simply a $\Gamma$-orbit, then requiring $R \toto L $ to be in ${\mathcal S} $ just amounts to require that $L \cap {\mathcal S}$ is a dense subset of $L$, for the transversality condition is automatically satisfied.  A Lie subgroupoid $R \toto L $ integrating an algebroid crossing, as defined in \cite{desing}, is automatically surjective in $\overline{\mathcal S} $.
\end{rmk}

We now define Morita equivalence of subgroupoids. 

\begin{defn}
Let $R \toto L $ be a Lie subgroupoid of $\Gamma \toto M $, and  $R' \toto L' $ be a Lie subgroupoid of
$\Gamma' \toto M' $. A Morita equivalence between these subgroupoids is given by a pair $({\mathcal X}, Y) $, where:
\begin{enumerate}
\item ${\mathcal X} $ is a Morita equivalence between the Lie groupoids $\Gamma \toto M $ and $\Gamma' \toto M' $,
\item ${\mathcal Y} $ is a Morita equivalence between the Lie groupoids $R      \toto L $ and $R'      \toto L' $,
\item an injective immersion ${\mathfrak i}:Y \hookrightarrow X $ such that the following diagram  commutes:
$$ \xymatrix{ & X \ar[dr]_{p'} \ar[dl]^{p} &   \\ M&     & M' \\   &  Y \ar[ur]^{p'} \ar@{^{(}->}[uu]^{\mathfrak i} \ar[ul]_{p} &   } $$
and which is compatible with the $R$ and $R'$-actions, i.e.
  $$ {\mathfrak i}(r \cdot y \cdot r') = r \cdot {\mathfrak i}(y) \cdot r' $$
for all $r \in R, y \in Y, r' \in R'$ s.t. $t(r)=p(y) $ and $p'(y) = s(r')$.  
\end{enumerate}
\end{defn}

\begin{rmk}
We warn the reader that, given a Morita equivalence ${\mathcal X} $ between $\Gamma \toto M$ and $\Gamma' \toto M' $, and given a Lie subgroupoid $R' \toto L' $ of $\Gamma' \toto L'$, there may not exist a subgroupoid $R \toto L $ of $\Gamma \toto L $ Morita equivalent to the first one. 
\end{rmk}

Morita equivalence of Lie subgroupoids can be composed.

\begin{prop}
Assume that we are given, for $i=1,2,3$, a Lie subgroupoid $R_i \toto L_i$ in $\Gamma_i \toto M_i) $. 
Let
\begin{enumerate} 
\item  $({\mathcal X} ,{\mathcal Y},{\mathfrak i})$   be a Morita equivalence  of Lie subgroupoids between $(R_1 \toto L_1,\Gamma_1 \toto M_1) $ and  $(R_2 \toto L_2,\Gamma_2 \toto M_2) $
\item  $({\mathcal X}',{\mathcal Y}',{\mathfrak i}')$ be a Morita equivalence of Lie subgroupoids  between $(R_2 \toto L_2,\Gamma_2 \toto M_2) $ and $(R_3 \toto L_3,\Gamma_3 \toto M_3) $.
\end{enumerate}
Then $ ({\mathcal X}'',{\mathcal Y}'',{\mathfrak i}'') $ is a Morita equivalence of Lie subgroupoids between $(R_1 \toto L_1,\Gamma_1 \toto M_1) $ and $(R_3 \toto L_3,\Gamma_3 \toto M_3) $ where:
\begin{enumerate}
\item  ${\mathcal X}'' $ is the composition of the Morita equivalences $ {\mathcal X}$
and ${\mathcal X}' $ as defined in section \ref{Morita-resolution}(a).
\item  ${\mathcal Y}'' $ is the composition of the Morita equivalences $ {\mathcal Y}$
and ${\mathcal Y}' $ as defined in section \ref{Morita-resolution}(a).
\item The map ${\mathfrak i}'': Y'' \hookrightarrow M'' $ is given by:
 $$ {\mathfrak i}''\big([(y,y')]\big)= [{\mathfrak i}(y),{\mathfrak i}'(y')],$$
for all $y \in Y,y' \in Y'$ s.t. $p'(y)=p(y')$, where $[(\cdot,\cdot)]$ stands for the class of an element in $Y \times_{p',L_2,p} Y'  $ or $ X \times_{p',M_2,p} X' $ modulo the action of $R_2 $ or $\Gamma_2$ respectively. 
\end{enumerate}
\end{prop}
\begin{proof}
The only point that has to be checked is that ${\mathfrak i}''$ is an injective immersion. First, 
 $$ {\mathfrak i}''([(y_1,y_1')])= {\mathfrak i}''([(y_2,y_2')]) $$
 implies that there exists $\gamma ' \in \Gamma'$ with $y_1 \cdot (\gamma')^{-1} = y_2   $ and $\gamma' \cdot y_1'= y_2'$. Since the action $R_2$ is transitive on the fibers of $p:Y \to M_1$, it follows from the first of these identities that $ \gamma' \in R_2 $, which is tantamount to $[(y_1,y_1')]=[(y_2,y_2')] $. 
Said differently, when one considers $Y \times_{p',L_2,p} Y' $ as a submanifold of $X \times_{p',M_2,p} X'$,  
we obtain that for every $(y,y') \in Y \times_{p',L_2,p} Y'$, the intersection of the $\Gamma_2$-orbit through
 $({\mathfrak i} (y),{\mathfrak i}'(y')) \in Y \times_{p',L_2,p} Y'  \subset X \times_{p',M_2,p} X'$ with $Y \times_{p',L_2,p} Y'$ is the $R_2$-orbit of $(y,y')$. 
 This last assertion proves that $i''$ is injective, but also
  that ${\mathfrak i}'' $ is an immersion.
\end{proof}

There is a natural notion of isomorphism of Morita equivalence of subgroupoids. Again, the composition defined by the last proposition is associative up to isomorphism, which allows the following convention:

\begin{convention} \label{conv:morita-subgroupoi}
From now, we identify isomorphic Morita equivalences of Lie subgroupoids, so that composition of Morita equivalences of Lie subgroupoids
is associative.

Also, given a Morita equivalence $ ({\mathcal X},{\mathcal Y},{\mathfrak i})$, we shall most of the time
consider $Y$ as a submanifold of $X$ (and therefore identify $y \in Y$ with $ {\mathfrak i}(y) \in X$). 
\end{convention}

 A {\em differential substack}  is an equivalence class of Lie subgroupoids modulo Morita equivalence. 

\begin{rmk}
Again, Lie subgroupoids do not form a set, so that it is a bit abusive a speak of "equivalence class".
\end{rmk} 
 
 We will in general consider differential substacks of a {\em given} stack $[\Gamma]$. We do it as follows. Given a Lie groupoid $\Gamma $, a {\em representative of a substack} of $[\Gamma]$ is a triple $(\Gamma', R',{\mathcal X})$ where $ {\mathcal X}$ is a Morita equivalence between $\Gamma$ and $\Gamma'$, and $R' \toto L' $ is a subgroupoid of $\Gamma' \toto M' $.    We say that two representatives $(\Gamma_1', R_1', {\mathcal X}_1)$ and $(\Gamma_2', R_2',{\mathcal X}_2)$ are Morita equivalent  if there exists a Morita equivalence $({\mathcal Z},{\mathcal Y},{\mathfrak i}'') $ between the subgroupoids $R_1'  $   and $R_2' $ such that 
$${\mathcal Z} = {\mathcal X}_2 \circ {\mathcal X}_1^{-1}.$$
By construction,  representatives of a substack of $[\Gamma]$ modulo
Morita equivalence (of representatives) form differential substacks. 

\begin{prop}\label{prop:Moritasubgrou}
Let $({\mathcal X},{\mathcal Y}, {\mathfrak i})$ be a Morita equivalence of subgroupoid between a subgroupoid $R \toto L $ of $\Gamma \toto M $ and  $R' \toto L' $ of $\Gamma' \toto M'$. Let ${\mathcal S} $ be a $\Gamma$-stable submanifold of $M$, and
$ {\mathcal S}' = \underline{\mathcal X }({\mathcal S})$ the corresponding $\Gamma'$-stable submanifold in $M'$.

The subgroupoid $R \toto L $ is in $ \overline{\mathcal S} $/ surjective in $\overline{\mathcal S}$/  full in $ {\mathcal S} $/ proper if
and only if  $R' \toto L' $ is in $ \overline{\mathcal S}' $/ surjective in $\overline{\mathcal S}'$/  full in $ {\mathcal S} '$/ proper.
\end{prop}
\begin{proof}
Morita equivalences being invertible, it suffices to show one direction, the proof of which is divided in the four claims below.

{\vspace{0.3cm}}
{\em Claim 1: If $R \toto L$ is in $\overline{\mathcal{S}} $, then 
$R'\toto L'$ is in $\overline{\mathcal{S}'} $.}

\noindent
First, we have to check that $L' \cap {\mathcal S}'$ is dense in $L'$.
By assumption, $ L \cap {\mathcal S} $ is dense in $ L$, so that, since $p_{|_Y}$
is a submersion, $ p_{|_Y}^{-1} ( L \cap {\mathcal S} ) $ is dense in $ Y$. In turn, this implies that 
$ p_{|_Y }' \big(p_{|_Y}^{-1} ( L \cap {\mathcal S} ) \big) $ is dense in $ p_{|_Y}'(Y) $.
But $$ p_{|_Y }' \big(p_{|_Y}^{-1} ( L \cap {\mathcal S} ) \big) =  L' \cap {\mathcal S}'  \mbox{ and } 
p_{|_Y}'(Y) = L' ,$$ so that $ L' \cap {\mathcal S}' $ is dense in $ L'$.

Second, we have to check that $L'$ intersects transversally all the $\Gamma $-orbits contained in ${\mathcal S}' $.
That it intersects all the orbits is clear: only the transversality condition requires a justification
Choose an arbitrary $m' \in L' \cap {\mathcal S}' $, and let $y \in Y $ an element with
$p'(y) = m' $. Let $m = p(y) $. Every tangent vector $u \in T_{m'} {\mathcal S}'$ is a derivative at $\epsilon=0 $ 
of a path $\epsilon \mapsto m'(\epsilon)$ in ${\mathcal S}' $. There exists a path $\epsilon \mapsto x(\epsilon)$ in $X$ starting at $y$ that projects on the path $\epsilon \mapsto m'(\epsilon)$
though $p'$, since $p'$ is a submersion.
 Since the path $\epsilon \mapsto p \circ x (\epsilon) $ is a path in ${\mathcal S}= p((p')^{-1}({\mathcal S}')) $ that
starts at $m=p(y) $, and since $T_m {\mathcal S} = T_m L + T_m F_m  $, there exists (according to remark \ref{rmk:trans}) a path $\epsilon \mapsto \gamma(\epsilon)$,
starting at the unit element $m \in \Gamma$,
in $\Gamma_L$ such that 
$$ t(\gamma (\epsilon)) =  p \circ x (\epsilon) $$
for all $\epsilon$ small enough. The path
  $$\epsilon \mapsto  \gamma (\epsilon) \cdot x(\epsilon)  $$
is in $p^{-1}(L) \subset X$, 
for all $\epsilon$ small enough.
There exists therefore a path $\epsilon \mapsto y(\epsilon)$
 in $Y $, starting at $y$, which projects on $s(\gamma(\epsilon))=p( \gamma (\epsilon) \cdot x(\epsilon) ) $ through $p$. 
Since $\Gamma' $ acts transitively on the fibers of $p$, there exists a path $\epsilon \mapsto \gamma'(\epsilon)  $ in $\Gamma'$ (starting at the unit element $m'\in M'$) s.t. 
  $$ \gamma(\epsilon)\cdot x(\epsilon) = y(\epsilon) \cdot \gamma'(\epsilon)  $$
for all $\epsilon$ small enough. Note that $\epsilon \mapsto \gamma'(\epsilon)$ takes in fact its values in $\Gamma_{L'}$. Applying $p' $ to the previous equality amounts to
 $$ p'(\gamma(\epsilon)\cdot x(\epsilon))  =  t(\gamma'(\epsilon)) ,$$
 for all $\epsilon$ small enough. Since the first term is equal to $m'(\epsilon) $, 
taking the derivative at $\epsilon=0$, we obtain (having in mind remark \ref{rmk:trans}) that $u \in T_{m'} L' + T_{m'} F_{m'}' $, where $F_{m'}$ is the $\Gamma'$-orbit through $m'$. Hence, $T_{m'} {\mathcal S}'=T_{m'} L' + T_{m'} F_{m'}' $, which completes the proof of the first claim.

{\vspace{0.3cm}}
{\em Claim 2:
If  $R \toto L$ is surjective in $\overline{\mathcal{S}} $, then $R' \toto L'$ is surjective in $\overline{\mathcal{S}'}$}.

\noindent
 Let ${\mathcal T}'  $ be a $\Gamma$-orbit contained in $\overline{\mathcal{S}'}$, and ${\mathcal T}  = \underline{\mathcal X}^{-1}( {\mathcal T}'  ) $. By Lemma \ref{lem:comm_clos}, or remark \ref{rmk:incl}, ${\mathcal T} $ is contained in 
 $\overline{\mathcal{S}}$. By assumption therefore, $ {\mathcal T}  \cap L $ is not empty. Since $p: Y\to L$ is onto, there exists $ y \in Y$, with $p(y) \in {\mathcal T}\cap L $. Hence $L' = p_{|_Y}' (p_{|_Y}^{-1}(L)) $ contains the element $p'(y)$.
But this element also belongs to ${\mathcal T}'= p'(p^{-1}({\mathcal T}))$, so that the intersection of $L'$ with ${\mathcal T}' $ is not empty. This conclusion being valid for an arbitrary $\Gamma$-orbit contained in $\overline{\mathcal{S}'}$, $R'\toto L'$ is surjective $\overline{\mathcal{S}'} $.

{\vspace{0.3cm}}
{\em  Claim 3: If  $R \toto L$ is full in ${\mathcal S} $, then $R' \toto L'$ is full in ${\mathcal S}' $ }. 

\noindent
Let $\gamma' \in \Gamma' $ be an arrow with source and target $
m_1'\in L' \cap {\mathcal S} $ and $ m_2'\in L' \cap {\mathcal S}$ respectively.
There exists $y_1,y_2 \in Y$ with $p'(y_1)=m_1', p'(y_2)=m_2'$. The relation
$p'(y_1 \cdot \gamma'  ) = m_2'= p'(y_2) $ holds true, hence there exists $\gamma \in \Gamma$
 with $ \gamma \cdot y_2 = y_1 \cdot \gamma'  $. Since both the source and target of $\gamma $ 
 are in ${\mathcal S} \cap L $ by construction, $\gamma$ belongs to $R $, so that $\gamma \cdot y_2 $ belongs to $Y $, and there exists therefore
 $r' \in R' $ with $ \gamma \cdot y_2 = y_1 \cdot r'$. By definition of Morita equivalence, the right action is
 free and $r'= \gamma'$. In particular, $\gamma' $ belongs to $R' $, and $R' \toto L'$ is full in
 ${\mathcal S} $.

{\vspace{0.3cm}}
{\em  Claim 4: If  $R \toto L$ is proper, then $R' \toto L'$ is proper.} 
This claim is left to the reader.
\end{proof}

Proposition \ref{prop:Moritasubgrou} justifies the following definition.

\begin{defn}
A substack is said to be in $ \overline{\mathcal S} $/ surjective in $ \overline{\mathcal S} $ /full in $ {\mathcal S} $/proper if and only if one (equivalently all) of its representatives is.
\end{defn}

\section{The correspondence between equivariant resolutions and substacks}
\label{the_correspondence}

 The purpose of this section is to show the main result of the present study, namely Theorem \ref{fund:theo}, which states the existence and describes the natural one-to-one correspondence between $\Gamma$-resolutions of $\overline{\mathcal S} $ and substacks of $[\Gamma ] $ full in ${\mathcal S} $.

We divide the construction of this correspondence in several steps. In section \ref{sec:subgrou_to_reso}, we associate  a $\Gamma$-resolution of $\overline{{\mathcal S}}$  to a closed subgroupoid $R\toto L $ of $\Gamma\toto M $ full in $ {\mathcal S}$.  This resolution is shown to be surjective (resp. proper) if the subgroupoid is.

Then in section \ref{sec:reso_to_subgrou}, we associate  to a $\Gamma$-resolution of  $\bar{\mathcal S} $ a closed subgroupoid $R' \toto L' $ of a Lie groupoid $\Gamma' \toto M' $ Morita equivalent to
$\Gamma \toto M $: more precisely we construct a representative
 $(\Gamma',R',{\mathcal X}) $, with ${\mathcal X} $ a Morita
 equivalence from $\Gamma \toto M $ to $\Gamma' \toto M' $, and $ R'
 \toto L'$ a closed subgroupoid of $\Gamma' \toto M' $  full in $ {\mathcal S}' =\underline{\mathcal X} ({\mathcal S}) $.
This subgroupoid is shown to be surjective in $\overline{{\mathcal S}'} $ (resp. proper) if  the resolution is 
surjective in $\overline{{\mathcal S}} $ (resp. proper).

 These constructions are {\em not} inverse to each other. However, we show in section \ref{sec:maintheorem} that they become inverse to each other, when we go down at the level of differential stacks, by taking the quotient of the whole picture by  Morita equivalence.

\subsection{From a subgroupoid to an equivariant resolution.} \label{sec:subgrou_to_reso}

We start by a proposition, a proof of which is presented in
\cite{desing} in the case where ${\mathcal S} $ is the Lie algebroid
orbit of an integrable Lie algebroid. The proof presented follows more
or less the same lines, but is much more general.

\begin{convention}
For every left-module $(X,\phi) $ over a Lie groupoid $\Gamma \toto M$, we denote by $\Gamma \backslash X $ the quotient space, i.e. the set obtained by identifying $x \in X$ with $\gamma \cdot x \in X$ for all $x \in X$, $\gamma \in \Gamma$ s.t. $t(\gamma)=\phi (x)$.
\end{convention}

The next proposition is of crucial importance.

\begin{prop} \label{prop:LGtoSt}
Let $ \Gamma \toto M$ be a Lie groupoid, ${\mathcal S} $ a $\Gamma$-stable submanifold in $M$, and $R \toto L $ a
subgroupoid of $\Gamma \toto M $  full in~${\mathcal S}$.
 \begin{enumerate}
 \item $Z(R)=  R \backslash \Gamma_L$ is a manifold,
 \item  there exists an unique smooth or holomorphic map $\phi: Z(R) \to M $ such that the following diagram commutes
\begin{equation}\label{eq:comdia1}
 \xymatrix{ \Gamma_L   \ar[r]^{p} \ar[rd]^{t}  & Z(R) \ar@<0.5ex>[d]^{\phi} \\ & M}
\end{equation}
Moreover, the map $\phi$ takes values in $\overline{\mathcal S}$.
\item $(Z(R),\phi) $ is an equivariant resolution of $\overline{\mathcal S} $,
\item if $R\toto L $ is surjective in $\bar{\mathcal S} $, then $(Z(R),\phi)$ is a surjective resolution,
\item if $R\toto L $ is proper, then $(Z(R),\phi)$ is a proper resolution.
\end{enumerate}
\end{prop}
\begin{proof}
1)  Since the source map $s$ is a surjective submersion from $\Gamma$ onto $M$, and $L $ a submanifold of $M$,
$\Gamma_L = s^{-1}(L) $ is a submanifold of $\Gamma$, acted upon on the left by $R$. We have to check that the quotient space
$R \backslash \Gamma_L $ is a manifold again.   
For that purpose, we first show that the left action of $R$ on $\Gamma_L$ is proper. Let  $(r_n, \gamma_n)  \in R \times_{t,L,s} \Gamma_L $ be a sequence such that $ (r_n\cdot \gamma_n , \gamma_n)$ takes values in a compact subset $K$ of $ \Gamma_L \times \Gamma_L$. By assumption, one can extract a subsequence $(r_{\sigma(n)} \times  \gamma_{\sigma(n)}, \gamma_{\sigma(n)}) $ that converges to $( g,g') \in K \subset \Gamma_L \times \Gamma_L$, so that $
r_{\sigma(n)}   $ converges to $r= g' \cdot g^{-1}  \in \Gamma$. We have to show that $r $ belongs to $R$. 
The subset $K$ being a compact subset, the  image of $\Gamma_L  \times \Gamma_L$ though
the maps $(x,y) \mapsto s(x) $ and $ (x,y) \mapsto s(y)$ are compact subsets $K_1$ and $K_2$ of $L$.
Since $s(r_n) \in K_1  $ and $t(r_n) = s(\gamma_n) \in K_2 $ for all $ n  \in {\mathbb N}$, and since $K_1 $ and $K_2$  are compact subsets, the source (resp. target) of $r  $ belongs to $K_1 $ (resp. $K_2$), hence both source and target belong to $L $.
In conclusion, $r$ belongs to $\Gamma_L^L $, and, since $R$ is closed in $ \Gamma_L^L$, we obtain that $r \in R$.
We eventually obtain that $r$ is an element in $ R$. In conclusion the left action of $R \toto L $ on $\Gamma_L $ is  a proper free action, so that the quotient space $R \backslash \Gamma_L $ is a manifold.
This completes the proof of 1).
2) A map $\phi $ satisfying Eq. (\ref{eq:comdia1}) always exists since the target map is not affected by left action of $R \toto L $ on $\Gamma_L $.  Since the canonical projection $\Gamma_L \to R \backslash \Gamma_L $  is a submersion, the map $\phi$ satisfying (\ref{eq:comdia1}) is  unique. Moreover, since the canonical projection $\Gamma_L \to R \backslash \Gamma_L $  is also a submersion and therefore admits local sections, the map $\phi$ is smooth or holomorphic, depending on the context. Since $L \cap {\mathcal S} $ is dense in $L $ and since the source map is a submersion, and therefore admits local sections, $\Gamma_{L \cap {\mathcal S}}$ is dense in $\Gamma_L $. Hence $ t(\Gamma_{L \cap {\mathcal S}})= {\mathcal S} $ is dense in  $t( \Gamma_{L} ) = \phi ( Z(R)  ) $. Hence 
$\phi ( Z(R)  ) \subset \overline{\mathcal S}  $.
This completes the proof of 2).
3) First,  since $L $ has by assumption a non-empty intersection with all the 
$\Gamma$-orbits contained in ${\mathcal S} $, the restriction of the
target map $t$ to $\Gamma_{L \cap {\mathcal S}} $ is a 
surjection onto ${\mathcal S} $, hence so is~$\phi: \phi^{-1}(\mathcal S) \mapsto  {\mathcal S}$.
Second, $R \toto L $ being a full subgroupoid of $\Gamma \toto M$, then,
for all $\gamma,\gamma' \in \Gamma_L $, the relation $t(\gamma) = t(\gamma') $ implies 
$$\gamma' \gamma^{-1}  \in \Gamma_{L \cap {\mathcal S} }^{L \cap {\mathcal S}} = R_{L \cap {\mathcal S}}^{L \cap {\mathcal S}} ,$$
hence $\gamma $ and $\gamma' $ define the same element in $Z(R) = R \backslash \Gamma_L $.
The restriction of $\phi :  \phi^{-1}(\mathcal S) \mapsto  {\mathcal S} $ is an injective map. In conclusion, $\phi: \phi^{-1}(\mathcal S) \mapsto  {\mathcal S} $ is a bijection.

\vspace{0.2cm}
\underline{Claim} The restriction $\phi: \phi^{-1}({\mathcal S})  \to {\mathcal S}$ is a biholomorphism / diffeomorphism. 
\vspace{0.2cm}

To show this point, it suffices to check that it is a surjective submersion, since we already know that it is a bijection.
First, since ${\mathcal S}$ is $\Gamma$-stable, $t^{-1} ({\mathcal S}) = \Gamma_{\mathcal S} $. Since moreover
the target map $t $ is  a surjective submersion from $\Gamma  $ to
$M$, the restriction to $\Gamma_{\mathcal S} $ 
of the target map is a surjective submersion from $ \Gamma_{\mathcal S}$ onto~${\mathcal S} $.

Second, for all $m \in L \cap {\mathcal S} $, the image of $T_m\Gamma_L $ through the differential  $T_mt $
of the target map at $m $ is the vector space $T_mL + T_m F_{m} $ ($F_m$ being the $\Gamma$-orbit through $m \in M$),
which is precisely assumed to be equal to $T_m {\mathcal S} $ by transversality,  so
that the target map is a surjective submersion from a neighborhood of $m \in \Gamma_L $ to a neighborhood of $m \in {\mathcal S} $.

Let us choose a point ${\mathfrak s} \in {\mathcal S} $, a tangent vector $u \in T_{\mathfrak s} {\mathcal S}
$ corresponding to an infinitesimal path $\epsilon \mapsto {\mathfrak s}( \epsilon )$. 
For every $\gamma \in \Gamma_{L \cap {\mathcal S}} $ with $t(\gamma)={\mathfrak s} $. Since the restriction to $\Gamma_{\mathcal S} $ 
of the target map is a surjective submersion  onto ${\mathcal S} $,
there exists an infinitesimal path  $\epsilon \mapsto \gamma(\epsilon) $
starting at $\gamma $ that projects on $\epsilon \mapsto {\mathfrak s}(\epsilon) $ through $t$. 
According to remark \ref{rmk:trans},  there exists an infinitesimal path
$\epsilon \mapsto \tilde{\gamma}(\epsilon) $ in $\Gamma_L $ starting at $m = s(\gamma) \in M $ such that 
  $$t(\tilde{\gamma}(\epsilon)) = s(\gamma(\epsilon)) $$
for all $\epsilon $ small enough. The infinitesimal path $\epsilon \mapsto \tilde{\gamma}(\epsilon)^{-1} \cdot \gamma (\epsilon) $
is well-defined, for all $\epsilon$ small enough, and is contained in $\Gamma_L $ by construction.
By construction also, it starts at $\gamma $ and its image through the target map is equal to the path $\epsilon \mapsto {\mathfrak s}(\epsilon)$, i.e. is an infinitesimal path that corresponds to $u $. Hence the differential of the restriction to $\Gamma_L $ of the target map is surjective, which proves the claim, and completes the proof of the fact that $(Z(R),\phi) $ is a resolution.

Last, the right action of $\Gamma \toto M$ on $(\Gamma_L,t)$ goes to the quotient and defines a right-$\Gamma$ action of $\Gamma \toto M$
on $(Z(R),\phi)$, hence this resolution is equivariant.
4) Now, if $R \toto L $ is moreover assumed to be surjective in
   $\overline{\mathcal S}$, then $L$ intersects all the groupoid leaves contained in $\Gamma $, and
the restriction of the target map $t$ to $\Gamma_{L } $ is a surjection onto $\overline{\mathcal S} $, hence so is~$\phi: Z(R) \mapsto  \overline{\mathcal S}$.
5) is straighforward, for the inverse image of a compact subset $K
   \subset \overline{\mathcal S} $ is precisely $R \backslash \Gamma^K_L $.
\end{proof}

We finish this section with a characterization of $\Gamma
$-resolutions of the type presented in \ref{prop:LGtoSt}, a point that shall be strongly useful in the next section. We start with a definition.


\begin{defn}\label{def:Lcomp}
Let ${\mathcal S} $ be an embedded $\Gamma$-stable submanifold of $M$, where
$\Gamma \toto M$ is a Lie groupoid.
 Let  $L \subset M $ be a submanifold  with $L \cap {\mathcal S}$ dense in $L$.
 A $\Gamma$-resolution $(Z, \phi) $ of $\overline{\mathcal S} $ is said to be {\em $L$-compatible} if
\begin{enumerate}
\item  there exists a submanifold $\tilde{L} $ such that the  restriction $\phi_{|_{\tilde{L}}}$ of $\phi $ to $\tilde{L} $ is a  
 biholomorphism / diffeomorphism from $\tilde{L} $ to $L$,
\item $\tilde{L} $ is transverse to the foliation on $Z$ given by the $\Gamma $-action, and intersects all the $\Gamma$-orbits 
contained in $\phi^{-1}({\mathcal S}) $
\end{enumerate}
\end{defn}

\begin{rmk}
Notice that $L$ has to be contained is $\overline{\mathcal S}$, and that, when $L$ is a given submanifold with $L \cap {\mathcal S}$ dense in $L$, $\tilde{L}$ is unique when it exists.
\end{rmk}

\begin{rmk}
 Notice that when ${\mathcal S} $ is an algebroid leaf, and $L $ intersects all the orbits
contained in $\overline{\mathcal S} $, $L $ is easily proved to be what is called in \cite{desing} an algebroid crossing. 
\end{rmk}

\begin{rmk} \label{LAaction}
To a Lie groupoid action of $\Gamma \toto M$ on a right-module $(Z,\phi)$  is associated a Lie algebroid action, i.e. a map
$\chi:A_m \to T_zZ $, for all $m \in \phi (Z), z \in Z $ s.t. $\phi (z) =m$, which induces a Lie algebra morphism
from the space $\Gamma (A) $ of sections of $A$ to the Lie algebra of vector fields on $Z$. 
The transversality assumption in the previous definition means that:
 $$ T_{\tilde{l}} Z = T_{\tilde{l}} L + \chi (A_{\phi(\tilde{l})})  $$
 for all $\tilde{l} \in \tilde{L}$.
\end{rmk}

\begin{example} \label{exampleOfLcomp}
Let $R \toto L $ is a closed Lie subgroupoid of $\Gamma \toto M$ full in ${\mathcal S} $. Then the equivariant resolution
$(Z(R),\phi)$ is $L$-compatible, the manifold $ \tilde{L}$ being in fact the image of $\epsilon(L) \subset \Gamma_L $ (recall that
$\epsilon: M \hookrightarrow \Gamma$ stands for the unit map) in $Z(R) = R \backslash \Gamma_L $. 
\end{example}

The previous example is almost the unique possible one,
as shown by the following proposition.

\begin{prop}
\label{prop:subgroupoid2}
Let $(Z,\phi)$ be an $L$-compatible $\Gamma$-resolution of $\overline{\mathcal S} $,
for some submanifold $L \subset M$ with $L \cap {\mathcal S} $ dense in $L$.
Let $\tilde{L} $ be the (unique) submanifold as in Definition \ref{def:Lcomp} and $\psi: L \to \tilde{L}$ the inverse of the
restriction of $\phi$ to $\tilde{L} $. 
Then,
\begin{enumerate}
\item  The set $R \subset L $ of all arrows $r \in \Gamma_L $ such that
   $$  \psi \big( s(r)\big)  \cdot r   \in \tilde{L}$$
is a Lie subgroupoid  of $\Gamma \toto M $ full in ${\mathcal S} $.
\item the $\Gamma$-orbit of $\tilde{L} $ in $Z$ is an open subset of
  $Z$  which coincides (as  a $\Gamma $-resolution) to the resolution $(Z(R),\phi_R)$.
\item if, moreover, $\tilde{L} $ has an intersection with all the $\Gamma$-orbits of $Z$, then the resolutions $(Z,\phi)$ and  $ (Z(R),\phi_R) $  coincide.
\end{enumerate}
\end{prop}
\begin{rmk}
We recall from convention \ref{con:isom} that we identify isomorphic 
$\Gamma$-resolutions (which is justified by the fact that the isomorphism between two isomorphic resolutions is unique when it exists),
so that the reader shall not surprised when we say that resolutions "coincide", and not simply that they are "isomorphic".
\end{rmk}
\begin{proof}
1) For all $r \in R$, the point $ \psi \big( s(r)\big)  \cdot r $ is contained in $ \tilde{L}$, and therefore has to be equal to $ \psi \big( t(r) \big)$. In particular, both the source and target on an element in $R$ are in $L$. It is then clear from the definition that $ R \toto L$ is a subgroupoid of $\Gamma \toto M $. Also, $R$ is a closed subset of $\Gamma_L^L$ by its very construction.

Our next task is to prove that it is a Lie subgroupoid: since $L$ is a submanifold, all we  need to prove in that $R $
is a submanifold as well. We do this by considering the map $\Psi: \Gamma_L \to Z$ given by 
$$\Psi :\gamma
\mapsto \psi \big( s(\gamma) \big) \cdot \gamma .$$
 By construction $R = \Psi^{-1}(\tilde{L}) $, so that it suffices, in order to ensure that $R $ is a submanifold, to prove that $\Psi$ is a submersion. 

We choose some arbitrary $\gamma \in \Gamma_L$  and $u \in T_z Z $, where $z=\Phi(\gamma)$. For every infinitesimal path $z(\epsilon) $ corresponding to $u$, $ \phi (z(\epsilon)) $ is an infinitesimal path corresponding to $T_z \phi (u) $. Since the target map $t $ is a submersion from $\Gamma$ to $M$, there exists a path $\gamma(\epsilon) \in \Gamma $ starting from $\gamma$ and whose image through $t$ is $ \epsilon \mapsto \phi (z(\epsilon)) $. 
The path $z(\epsilon ) \cdot  \gamma^{-1}(\epsilon)  $ is an
infinitesimal path starting from $\tilde{l}:=\psi (s(\gamma) )  $.

The transversality condition implies that  $\Psi $ is a submersion in a neighborhood of
$\epsilon (L) \subset \Gamma_L$. In particular, it is a submersion in a
neighborhood of $  \tilde{l} = \psi (s(\gamma) ) $, and there exists an
infinitesimal path $ \epsilon \mapsto \tilde{\gamma}(\epsilon) \in \Gamma_L $ such that  
$$ \Psi \big( \tilde{\gamma}({\epsilon})\big) = z(\epsilon ) \cdot  \gamma^{-1}(\epsilon)   ,$$
for all $\epsilon  $ small enough. The latter can be rewritten as:
 $$ \Psi \big( \tilde{\gamma}({\epsilon})  \gamma (\epsilon )\big) = z(\epsilon ) $$
But $ \epsilon \mapsto \tilde{\gamma}({\epsilon})  \gamma (\epsilon ) $ is a path in $ \Gamma_L$ par construction. Taking the derivative at $
\epsilon = 0$, we obtain:
 $$  {\rm d}_l \Psi \left( \left.  \frac{{\rm d} \,
 \tilde{\gamma}({\epsilon})  \gamma (\epsilon ) }{{\rm d} \epsilon }
 \right|_{\epsilon =0} \right) = u .$$
Hence $\Psi  $ is a submersion, and $R \toto L $ is a Lie subgroupoid of $\Gamma \toto M$.

We now have to check that this Lie groupoid is in ${\mathcal S} $.
Since the restriction to $\phi^{-1}({\mathcal S}) $ of $\phi $ is invertible (and $\Gamma$-equivariant), it is immediate that
$L $ intersects transversally all the $\Gamma$-orbits contained in ${\mathcal S}$, since, by assumption, 
$\tilde{L} $ intersects transversally all the $\Gamma$-orbits contained in $\phi^{-1}({\mathcal S})$.
By its very construction, this Lie groupoid satisfies 
$$R_{L\cap {\mathcal S}}^{L\cap {\mathcal S}}=\Gamma_{L\cap {\mathcal S}}^{L\cap {\mathcal S}},$$
i.e it is a full Lie subgroupoid. This completes the proof of 1).

2) For any pair  $\gamma ,\gamma' \in \Gamma$ defining the same
   element in $Z(R) $, i.e. such that there exists $r \in R $  with $\gamma =r \gamma' $,
   one computes: 
$$ \begin{array}{rcl} \Psi (\gamma ') =  \psi \big( s(\gamma') \big) \cdot \gamma' &=&  \psi \big( s(\gamma') \big) \cdot r^{-1} 
r\gamma' \\                                          &=&  \psi\big( s(r) \big) \cdot r \gamma'  \\ 
                                                     &=&  \psi \big(s(\gamma) \big)
                                                     \cdot \gamma \\
                                            &=& \Psi (\gamma),
                                        \end{array} $$ 
where the relation $ \psi \big( s(r') \big) \cdot r'=\psi(t(r'))$ for
all $r' \in R$ has been used. As a conclusion, the map $\Psi$ goes to the quotient and defines a map $\tilde{\Psi}$ 
from $ Z(R)$ to $Z $, which is a morphism of resolution, and whose image is by construction the orbit of $\tilde{L} $.
This map $\tilde{\Psi}$ is also injective since $ \Psi (\gamma) = \Psi (\gamma')$ implies that $\gamma$ and $ {\gamma'}^{-1}$ are 
compatible and that the following identities hold:
   $$ \psi  \big( s(\gamma) \big) \cdot   \gamma (\gamma')^{-1}  = \psi \big( s(\gamma') \big). $$  
which, in turn, since both $ \psi  \big( s(\gamma) \big)$ and $  \psi
\big( s(\gamma') \big)$
belong to $\tilde{L}$ , gives that $\gamma (\gamma')^{-1         }  \in R$ (this
is the very definition of $R$). Hence
$\gamma$ and $\gamma'$ define the same element modulo the $R$-action. Moreover, the map $ \tilde{\Psi}$ is again a submersion, since $\Psi $ is a submersion. Since an injective submersion is in fact an open immersion, this completes the proof of 2).

3) follows from the fact that the image of the map $\Psi$ is precisely the orbit of $ \tilde{L}$ under the action of $\Gamma$. 
\end{proof}

\subsection{From an equivariant resolution to a subgroupoid.}
\label{sec:reso_to_subgrou}

Let $ \Gamma \toto M$ be a Lie groupoid, ${\mathcal S} $ a $\Gamma$-stable submanifold in $M$,
and $(Z,\phi) $ a $\Gamma$-resolution of $\overline{\mathcal S} $. 

By the {\em direct product Lie groupoid  $(\Gamma \toto M ) \times (Z \times Z \toto Z)  $}, 
we mean the Lie groupoid structure on $\Gamma \times Z \times Z$ with unit manifold $M \times Z$, with unit map 
$(m,z) \mapsto (\epsilon(m),z,z) $, with source and target maps $s:(\gamma,z_1,z_2) \mapsto (s(\gamma),z_1)$ and $t:(\gamma,z_1,z_2) \mapsto (t(\gamma),z_2)$ respectively, with product
 $$ (\gamma,z_1,z_2)\cdot (\gamma',z_2,z_3) = (\gamma\gamma',z_1,z_3) $$
  (for all $\gamma,\gamma'$ with $t(\gamma)=s(\gamma')$
 and all $z_1,z_2,z_3 \in Z$)
 and with inverse map $(\gamma,z_1,z_2) \mapsto (\gamma^{-1},z_2,z_1)  $.
This groupoid structure is the direct product of the Lie groupoid $\Gamma \toto M$ with the pair groupoid
$Z \times Z \toto Z $, hence the name. 

There is a natural Morita equivalence ${\mathcal X} $ between the Lie groupoid $\Gamma \toto M $
and the direct product Lie groupoid  $(\Gamma \toto M ) \times (Z \times Z \toto Z)  $ defined 
as follows:
\begin{enumerate}
\item $X= \Gamma \times Z $
\item $p: X \to M $ is the map $(\gamma,z) \mapsto s(\gamma) $, while
  $p': X \to M \times Z$ is the map $\gamma,z \mapsto (t(\gamma),z)$,
\item the right and left actions given respectively by:
$$ \left\{ \begin{array}{rcll}
 \gamma' \cdot (\gamma,z) &= &(\gamma' \gamma, z') & \forall \gamma',
 \gamma \in \Gamma, z \in Z \mbox{ s.t. } t(\gamma') =s(\gamma) \\
 (\gamma,z) \cdot (\gamma', (z,z')) &=& (\gamma \cdot \gamma', z') 
 & \forall \gamma',
 \gamma \in \Gamma, z,z' \in Z \mbox{ s.t. } t(\gamma) =s(\gamma') \end{array}\right.  $$
\end{enumerate}
The submanifold $\underline{\mathcal X}({\mathcal S}) $ corresponding
to ${\mathcal S} $ through this Morita equivalence 
${\mathcal X} $ is $ {\mathcal S}'={\mathcal S} \times Z $. 
The resolution $\underline{\mathcal X}\big((Z,\phi)\big) $
corresponding to $(Z,\phi) $ is the resolution 
$( Z \times Z, \phi \times {\rm id}_Z ) $. The right action of
$(\Gamma \toto M) \times (Z \times Z \toto Z)  $ on $(Z \times Z,\phi
\times {\rm id}_Z)$ is given by:
 $$  (z, z_1) \cdot \big( \gamma, (z_1,z_2) \big)  = (z \cdot \gamma, z_2 )  , $$
for all $\gamma \in \Gamma$ and $ z,z_1, z_2 \in Z $ with $s(\gamma) = \phi(z) $.

\begin{convention}
We shall from now introduce the shorthands $\widehat{\Gamma} \toto \widehat{M} $ for $(\Gamma \toto M) \times (Z \times Z \toto Z) $,   $ \widehat{{\mathcal
  S}}$ for $   \underline{\mathcal X} ({\mathcal S}) =  {\mathcal  S} \times Z $,
  $(\widehat{Z},\widehat{\phi}) $ for $ \underline{\mathcal X} \big((Z,\phi) \big)  =( Z \times Z, \phi \times {\rm id}_Z ) $.
 \end{convention}

The reader should have in mind the previous conventions for a correct understanding of the coming 
proposition:
   
\begin{prop} Let notations be as in the lines before.
The resolution $(\widehat{Z}, \widehat{\phi}) $ is $L$-compatible, where
 $L = \{(\phi(z),z ) | z \in Z \} $.
\end{prop}
\begin{proof}
Let $\tilde{L} \subset \hat{Z} = Z \times Z$ be the diagonal.
The map $\widehat{\phi} = \phi \times {\rm id}$ restricts to a biholomorphism / diffeomorphism from $\tilde{L} $ to its image $L$,
which is a submanifold of $M \times Z$. 
Now,  $\tilde{L} $ is transverse to the  action of $\widehat{\Gamma} \toto \widehat{M} $
on $\widehat{Z}= Z\times Z $, since the tangent space at a point $(z,z) \in \tilde{L}$ of the leaves of the
$\Gamma'$ action always contain the space $\{ (0, u), u \in T_z M \}$, so
that its sum with the tangent space of the diagonal is $T_{(z,z)} (M \times \widehat{Z})$.
It clearly intersects all the $\widehat{\Gamma}$-orbit, since 
$(z,z') \in \hat{Z}=Z \times Z$ and $(z,z)$ are in the same $\widehat{\Gamma}$-orbit.
\end{proof}

Let $\widehat{R} \toto \widehat{L} $ be the Lie subgroupoid full in ${\mathcal S}' $ corresponding to $\tilde{L} $ as in Proposition \ref{prop:subgroupoid2} (1). The next corollary follows immediately  from Proposition 
\ref{prop:subgroupoid2} (3).

\begin{coro}
The resolution $(Z(\widehat{R}), \phi_{\widehat{R}}) $ associated to
$\widehat{R} \toto \widehat{L}$  is  the resolution $ (\widehat{Z} , \widehat{\phi})$ corresponding to $(Z,\phi) $ via the Morita equivalence $ {\mathcal X}$. In equation:
 $$  (Z,\phi) =  \underline{\mathcal X}^{-1}\big( (Z(\widehat{R}), \phi_{\widehat{R}}) \big) $$
\end{coro}
\subsection{The main theorem.}
\label{sec:maintheorem}

We start with a proposition, which means that "Full Lie subgroupoids
give isomorphic resolutions if and only if they are Morita
equivalent". 

\begin{prop}\label{fund:prop}
Let  ${\mathcal X} $ be a Morita equivalence between  $\Gamma \toto M $ and $\Gamma' \toto L' $, ${\mathcal S} $  a $\Gamma$-stable submanifold
 of $M $, and $ {\mathcal S}' = \underline{\mathcal X} ( {\mathcal S}
 ) $, $R \toto  L$ a subgroupoid full in ${\mathcal S} $, and $R ' \toto L' $ a subgroupoid full in ${\mathcal S} $'.

Then the following are equivalent:
\begin{enumerate}
\item[(i)] the resolutions  $\underline{\mathcal X}\big(( Z(R),\phi_R )\big) $ and $(Z(R'),
\phi_{R'})    $ coincide
\item[(ii)] there exists a Morita equivalence of Lie subgroupoids 
of the form $({\mathcal X},{\mathcal Y},{\mathfrak i}) $ between the
subgroupoids $R \toto L $ of $\Gamma \toto M$ and $R' \toto L' $ of
$\Gamma' \toto M' $.
\end{enumerate}
\end{prop}
\begin{rmk}
In fact, the proof will show that the pair $({\mathcal Y},{\mathfrak i})$ that appears in item (ii)
of the proposition is unique when it exists (up to isomorphism, see convention \ref{conv:morita-subgroupoi}).
Recall also from convention \ref{conv:morita-subgroupoi} that  a Morita equivalence of Lie subgroupoids 
of the form $({\mathcal X},{\mathcal Y},{\mathfrak i}) $ is in fact given by a submanifold
$Y$ of $X$. Our precise claim is that this submanifold is unique: this follows from step 4 in the proof below.
\end{rmk}
\begin{proof} 
We first prove that {\em (i)} and {\em (ii)}.
Assume that 
$$\underline{\mathcal X}\big(( Z(R),\phi_R )\big) =(Z(R'),\phi_{R'}). $$ 
Denote by $\tilde{L}, \tilde{L'}$ the submanifolds of $Z(R) $ and $Z(R') =\underline{\mathcal X}\big(( Z(R),\phi_R )\big)$ respectively, constructed as in Example \ref{exampleOfLcomp}, to which the restrictions of $\phi$ and $\phi' $ respectively is a biholomorphism / diffeomorphism onto $L$ and $L'$ respectively.

{\em Step 1. Construction of $Y$.}
Since $p$ is a submersion, and since the restriction of $\phi$ to $\tilde{L}$ is invertible,
 $$P:= \tilde{L} \times_{\phi,L,p} X$$
  is a submanifold of $Z(R) \times_{\phi_R,M,p} X$, and the projection onto the 
second component is in fact an isomorphism on its image, which is precisely $p^{-1}(L) \subset X$. 
Let $Y$ be the inverse image of $ \tilde{L'}$ through the restriction to $P$ of the natural projection:
 $$ \Pi : Z(R) \times_{\phi, M, p} X  \mapsto \frac{Z(R) \times_{\phi, M, p} X }{ (z,x)\sim (z
  \gamma^{-1}, \gamma x ) } .$$
The set $Y$ is by construction a subset of $P$, but since $P$ can be seen as a subset of $X$,
we may consider $Y$ as a subset of $X$. 
We have to show that  $Y \subset X$ gives a Morita equivalence 
between $R \toto L $ and $R' \toto L' $.
This fact is established in the steps 2,3 and 5 below.

{\em Step 2. $Y \subset X$ is a submanifold.} 
To prove this fact, it suffices to show that the restriction of $\Pi $ to $P$ is a submersion onto $Z(R')$,
which can be done as follows. Let $z' \in Z(R')$ be a point, $(z,x) \in \tilde{L} \times_{\phi,M,p}X $
with $\Pi (z,x)=z'$. Choose $u \in T_{z'} Z(R')$ a tangent vector,
and $\epsilon \mapsto z'(\epsilon) $ an infinitesimal path corresponding to $u$. 
Since the natural projection map $ Z(R) \times_{\phi,M,p} X \to Z(R')  $ is a submersion,
there exists an infinitesimal path, starting at $(z,x)$, 
  $$\epsilon \mapsto \big(z(\epsilon), x(\epsilon)\big) $$
in $Z(R) \times_{\phi,M,p} X \to Z(R') $ that project on $\epsilon \to z'(\epsilon)$. Since the natural projection $ \Gamma_L \mapsto R \backslash \Gamma_L $ is also a surjective submersion, there exists an infinitesimal path $\epsilon \mapsto \gamma (\epsilon) \in \Gamma_L$
which projects to $z(\epsilon)$. In particular, $z(\epsilon) \cdot \gamma^{-1}(\epsilon)$ belongs
to $\tilde{L}$ by construction for all value of $\epsilon$, so that
$$ \epsilon \mapsto \big(  z(\epsilon) \cdot \gamma^{-1} (\epsilon) , \gamma(\epsilon) \cdot x(\epsilon) \big)  $$
is a path in $P $.
The image through $P$ of this last path is $z'(\epsilon) $ again. Hence $P$ is a submersion, and $Y $ is a submanifold of $X$.

{\em Step 3. The action the $R$-action on $Y$.} 
It follows directly from the definition of $Y$ that, for every $y \in Y $, and every compatible $r \in R, r' \in R'$, $r \cdot y \cdot r'  $ is again an element in $Y$.  Since the left $\Gamma$-action  on $X$ is free, the left action of $R \toto L$ on $Y$ is again free. Since the left $\Gamma$-action  on $X$ is proper,  and since $R \toto L$ is a closed subgroupoid, the left action of $R \toto L$ on $Y$ is again proper. 

{\em Step 4. A characterization of $Y$.} 
 Since $\Pi$ is a submersion, the inverse image of $\tilde{L}' \cap {\mathcal S}$ through $\Pi$
is an open and dense subset of $Y$. This subset is clearly equal to 
$$X_{L \cap {\mathcal S}}^{L' \cap {\mathcal S}'}=  p^{-1}(L \cap {\mathcal S}) \cap (p')^{-1}(L' \cap {\mathcal S}') $$
 so that $Y$, being a closed subset of $X_L^{L'}:= p^{-1}(L) \cap (p')^{-1}(L') $ by construction, is in fact equal to the closure of $X_{L \cap {\mathcal S}}^{L' \cap {\mathcal S}'}   $ in $X_L^{L'}$. In equation:
 $$ .$$

{\em Step 4. The $R$-action is transitive on the fibers of $p'$.} 
 Since $R \toto L$ is full in ${\mathcal S} $, the action of $R \toto L$ is transitive on the fibers of 
 $$p': X_{L \cap {\mathcal S}}^{L' \cap {\mathcal S}'} \mapsto {L' \cap {\mathcal S}'} ,$$
since if $y^1,y^2 \in X_{L \cap {\mathcal S}}^{L' \cap {\mathcal S}'} $ satisfy $p'(y^1)=p'(y^2)$, 
 then the unique element $\gamma \in \Gamma$ with $y^1=\gamma \cdot y^2$ has its sources and targets
in $L \cap {\mathcal S}$, hence its belongs to $R$. Let us show that the $R$-action being proper, it has to be also transitive on all the fibers of $ p_{|_Y}' $. Let $y^1,y^2 \in Y$ be two elements with $p'(y^1)=p'(y^2) $. Since $p'_{|_Y}$ is a submersion, there exists sequences $(y^1_n)_{n \in {\mathbb N}} $, $(y^2_n)_{n \in {\mathbb N}} $ in $(p')^{-1} (L \cap {\mathcal S})$ that converge to $y^1$ and $y^2 $ respectively, and such that 
   $$ p'(y^1_n) = p'( y^2_n ) \mbox{ for all $n \in {\mathbb N}$} .$$ 
There exist a sequence $(r_n)_{n \in {\mathbb N}} $ s.t. $y^1_n=r_n \dot y^2_n $ for all $n \in {\mathbb N}$.
The action being proper, one can extract a subsequence of the sequence $(r_n)_{n \in {\mathbb N}}$ converges to an element $r \in R$
which satisfies $y^1 = r\dot y^2$.

{\em Step 5. The restriction of $p'$ is a submersion onto $L$, and the $R'$-action is free, proper,
and transitive on the fibers of $p'$.} 
In our way to prove that $Y$ gives a Morita equivalence between $R$ and $R'$, we have only obtained so far half of the requirements.
 The second half can be in fact obtained by symmetry of the whole picture. By inverting the roles of $L$ and $L'$ in the previous constructions, one obtains an other subset $Y' $ of $X$. More precisely, $Y'$ is the inverse image of $\tilde{L} $ through the restriction 
to $P' = X \times_{ p',L,\phi'} \tilde{L} $ (which can be seen as a subset of $X$) of the natural projection 
$$ Z \times_{\phi, M, p} X  \mapsto \frac{Z \times_{\phi, M, p} X }{ (z,x)\sim (z \gamma^{-1}, \gamma x ) } $$ 
As before, we can arrive at the conclusion that $Y'$ is the closure in $X_L^{L'}$ of $ X_{L \cap {\mathcal S}}^{L' \cap {\mathcal S}'}  $. In particular, we have $Y' =Y $. Therefore, since we have already proven that the restriction of $p'$ to $Y$ is a submersion onto $L'$, the fibers of which are given the free and proper $R$-action, we can conclude without additional effort, due to that symmetry, that the restriction of  $p$ to $Y$ is a submersion onto $L$, the fibers of which are given by the free and proper $R'$-action.

This completes the first part of the proof.
    
\vspace{1cm}
We now turn our attention to the other direction. Assume that {\em (ii)} is satisfied, i.e. that there exists 
a Morita equivalence $({\mathcal X},{\mathcal Y},{\mathfrak i}) $ between the subgroupoids $R \toto L $ of $\Gamma \toto M$ and $R' \toto L' $ of $\Gamma' \toto M' $, and that this Morita equivalence is given by a submanifold $Y \subset X$ (see convention
\ref{conv:morita-subgroupoi}).
 
For all $ (\gamma, x)  \in \Gamma_L \times_{t,M,p }X$, there exists an element $\gamma' \in \Gamma_{L'}'$ which satisfies that $\gamma \cdot x \cdot (\gamma')^{-1} \in Y  $. This element is not unique, but two of them differ by left multiplication by an element of $R'$.
Hence, we have a well-defined map:
 $$\Xi : \Gamma_L \times_{t,M,p }X \mapsto Z(R') .$$
Let us prove that this map is a submersion. Choose an arbitrary $z' \in Z(R')$ and $u \in T_{z'} Z(R')$. Let $(\gamma,x) \in  \Gamma_L \times_{t,M,p }X $ such that $\Xi (\gamma,x)= z $, and let $\epsilon \mapsto \gamma'(\epsilon)$ by a path in $\Gamma_{L'}'$ whose image through the natural projection onto $Z(R') $ is an infinitesimal path corresponding to $u$. Let $\epsilon \mapsto x (\epsilon) $ be an infinitesimal path in $X$ starting from $ x \in X $ such that $x(\epsilon) \cdot (\gamma'(\epsilon))^{-1}$
is well defined for all $\epsilon$ small enough. By construction, the path 
$$ \epsilon \mapsto p'\big( x(\epsilon) \cdot (\gamma'(\epsilon))^{-1} \big) $$   
takes in values in $L'$, so that there exists an infinitesimal path $\epsilon \mapsto \gamma (\epsilon) $
such that $ \gamma(\epsilon) \cdot x(\epsilon) \cdot (\gamma'(\epsilon))^{-1} $ is in $ Y$
for all $\epsilon $ small enough. By construction the path
 $$ \epsilon \mapsto  \Xi \big(\gamma(\epsilon) , x(\epsilon)\big) $$ 
is the path $\epsilon \mapsto [\gamma'(\epsilon)]$ (where $[\cdot]$ stands for the class of an element in
$\Gamma_{L'}' $ modulo $R'$), i.e. is an infinitesimal path that corresponds to $ u$. This completes the proof of the claim.

The submersion $\Xi$ goes to the quotient under the right-action of $R \toto L$ on $\Gamma_L $ to define a submersion 
  $$ Z(R) \times_{\phi,M,p }X \mapsto Z(R')  $$
 which, in turn,  goes the quotient with respect to the diagonal action of $\Gamma \toto M$ on 
 $Z(R) \times_{\phi,M,p }X $ to eventually define a submersion:
   $$ \underline{\mathcal X} \big( Z(R) \big) = \frac{Z(R) \times_{\phi,M,p }X}{ (z,x)\sim (z
  \gamma^{-1}, \gamma x ) }  \mapsto Z(R').$$
   This map can be easily checked to be one-to-one and equivariant w.r.t. the right $\Gamma'$-action. It is therefore an isomorphism of equivariant resolution.
\end{proof}

Altogether, Proposition \ref{fund:prop}, and the constructions given
in Sections \ref{sec:subgrou_to_reso} and \ref{sec:reso_to_subgrou}
amount to the following theorem.

\begin{theo}\label{fund:theo}
Let $\Gamma \toto M $ be a Lie groupoid, and ${\mathcal S} $ a
$\Gamma$-stable submanifold in $M$. There is a natural one-to-one
correspondence between $\Gamma $-resolutions of $\overline{\mathcal S}
$ and  substacks of $[\Gamma] $ full in ${\mathcal S} $.  Under this
correspondence, surjective resolutions corresponds to surjective
substacks and proper resolutions correspond to proper substacks. 
\end{theo}
\begin{proof}
A representative of a substack of $[\Gamma]$ full in ${\mathcal S} $ is by definition a pair
$(\Gamma',{\mathcal X},R') $ with ${\mathcal X} $ a Morita equivalence
from $\Gamma $ to $\Gamma' $ and $\overline{\mathcal S} $ a Lie subgroupoid of $\Gamma' \toto M' $ full in $\underline{\mathcal X}({\mathcal S})$.

To this representative of a substack of $[\Gamma]$, we assign the resolution
$$\underline{\mathcal X}^{-1} \big( Z(R'), \phi_{R'}\big) .$$ 

We now check that this assignment makes sense: let $(\Gamma_1',{\mathcal X}_1,R_1) $ 
and $(\Gamma_2',{\mathcal X}_2,R_2) $ be two representatives of the
  same substack of $[\Gamma] $, that is to say such that
there exits a  Morita equivalence of subgroupoids $({\mathcal X}',{\mathcal Y},{\mathfrak i}) $ between the subgroupoids $R_1'$ and $R_2' $
where ${\mathcal X}'= {\mathcal X}_2 \circ {\mathcal X}_1^{-1}$.
Then, according to Proposition \ref{fund:prop}, we have
$\underline{\mathcal X}'\big(  Z(R_1'),\phi_{R_1'} \big) = \big(  Z(R_2'),\phi_{R_2'} \big) $, or, equivalently,
 $$ \underline{\mathcal X}_1^{-1} \, \big(  Z(R_1'),\phi_{R_1'} \big) = 
\underline{\mathcal X}_2^{-1} \, \big(  Z(R_2'),\phi_{R_2'} \big).$$
In words, the previously defined assignment is compatible with respect to Morita equivalence and defines an assignment from substacks of $[\Gamma] $ full in ${\mathcal  S} $ to resolutions of $\overline{\mathcal S} $.
This assignment is injective, for, if 
  $$ \underline{\mathcal X}_1^{-1} \, \big(  Z(R_1'),\phi_{R_1'} \big) = 
\underline{\mathcal X}_2^{-1} \, \big(  Z(R_2'),\phi_{R_2'} \big),$$
then $ \underline{\mathcal X}'= \underline{\mathcal X}_2 \circ \underline{\mathcal X}_1^{-1} $  maps $ ( Z(R_1'),\phi_{R_1'})$
to  $(Z(R_1'),\phi_{R_1'})$, so that, by Proposition \ref{fund:prop} again, 
$R_1' $  and $R_2' $ are Morita equivalent Lie subgroupoids.

Now, in Section \ref{sec:reso_to_subgrou}, we have constructed, given
a resolution $(Z,\phi)$ a triple $(\Gamma',{\mathcal X},R') $ with $
\underline{\mathcal X}^{-1} \big((Z(R'),\phi) \big)=(Z,\phi) $, which
proves the surjectivity of the assignment. This completes the proof of
the first part of the theorem.  The second part follows from item 4)
and 5) in proposition \ref{prop:LGtoSt}
\end{proof}

\end{document}